\documentclass[review]{elsarticle}
\makeatletter
\def\ps@pprintTitle{%
 \let\@oddhead\@empty
 \let\@evenhead\@empty
 \def\@oddfoot{}%
 \let\@evenfoot\@oddfoot}
\makeatother

\usepackage{graphicx}
\usepackage{amssymb}
\usepackage{amsmath}
\usepackage{amsthm}
\usepackage{tkz-graph}
\usepackage[utf8]{inputenc}
\usepackage[T1]{fontenc}

\usepackage{lineno}

\newtheorem{thm}{Theorem}
\newtheorem{lem}[thm]{Lemma}
\newdefinition{rmk}{Remark}

\newproof{pf}{Proof}
\newproof{pot}{Proof of Theorem \ref{thm2}}

\begin{document}

\SetGraphUnit{1}
  \GraphInit[vstyle=Simple]
  \SetVertexSimple[MinSize    = 6pt, LineColor = blue!60, FillColor = blue!60]

\begin{frontmatter}


\title{Tur\'{a}n numbers of vertex-disjoint cliques in $r$-partite graphs}

\author[unl]{Jessica De Silva\corref{cor1}\fnref{fn1}}
\ead{jessica.desilva@huskers.unl.edu}

\author[macalester]{Kristin Heysse}
\ead{kheysse@macalester.edu}

\author[uw]{Adam Kapilow\fnref{fn2}}
\ead{akapilow@uw.edu}

\author[montana]{Anna Schenfisch\fnref{fn2}}
\ead{annaschenfisch@montana.edu}

\author[iastate]{Michael Young\fnref{fn2}}
\ead{myoung@iastate.edu}

\address[unl]{University of Nebraska-Lincoln, United States}
\address[macalester]{Macalester College, Saint Paul, MN, 55105, United States}
\address[uw]{University of Washington, United States}
\address[montana]{Montana State University, United States}
\address[iastate]{Iowa State University, United States}
\fntext[fn1]{Research is supported in part by the NSF-GRFP Grant DGE-1041000.}
\fntext[fn2]{Research was supported in part by the NSA Grant H98230-16-1-0038.}
\cortext[cor1]{Principal corresponding author}


\begin{abstract}
For two graphs $G$ and $H$, the \emph{Tur\'{a}n number} $\operatorname{ex}(G,H)$ is the maximum number of edges in a subgraph of $G$ that contains no copy of $H$. Chen, Li, and Tu determined the Tur\'{a}n numbers $\operatorname{ex}(K_{m,n},kK_2)$ for all $k\geq 1$ \cite{CTL}. In this paper we will determine the Tur\'{a}n numbers $\operatorname{ex}(K_{a_1,\ldots,a_r},kK_r)$ for all $r\geq 3$ and $k\geq 1$.
\end{abstract}

\end{frontmatter}


\section{Introduction}
All graphs considered here are finite, undirected, and simple. Throughout the paper we use the standard graph theory notation (see \cite{BM}). In particular, a graph is called a \emph{complete $r$-partite graph} if its vertex set can be partitioned into $r$ independent sets $V_1,\ldots,V_r$ such that for any $i=1,2,\ldots,r$ every vertex in $V_i$ is adjacent to all other vertices in $V_j$, $j\neq i$. We denote a complete $r$-partite graph with part sizes $|V_i|=n_i$ by $K_{n_1,\ldots,n_r}$. For a graph $G$ and a positive integer $k$ we use $kG$ to denote $k$ vertex-disjoint copies of $G$. Given $S\subseteq V(G)$, the subgraph of $G$ induced by $S$ will be denoted $G[S]$ and the subgraph $G[V(G)\backslash S]$ will be denoted $G\backslash S$. For two graphs $G$ and $H$, $G+H$ is the join of $G$ and $H$, that is the graph obtained from $G\cup H$ by adding every edge containing a vertex of $G$ and a vertex of $H$.

For two graphs $G$ and $H$, the \emph{Tur\'{a}n number}, or \emph{extremal number}, $\operatorname{ex}(G,H)$ is the maximum number of edges among all $H-$free subgraphs of a host graph $G$. The study of such numbers began in 1907 when Willem Mantel determined the maximum number of edges in a triangle-free graph on $n$ vertices, i.e. $\operatorname{ex}(K_n,K_3)$. In 1941, this theorem was strengthened by P\'{a}l Tur\'{a}n who determined $\operatorname{ex}(K_n,K_r)$. Since then, the most well-studied host graphs have been the complete graph $K_n$ and the complete bipartite graph $K_{m,n}$. Recent studies of extremal numbers consider the case when the forbidden graph $H$ is made up of several vertex-disjoint copies of some smaller graph (e.g., \cite{BK}, \cite{LLP}, \cite{Gor}, \cite{YZ}). In particular, Chen, Li, and Tu determined $\operatorname{ex}(K_{m,n},kK_2)=m(k-1)$ for $1\leq k\leq n\leq m$ \cite{CTL}. The focus of this paper is to extend their result to forbidding vertex-disjoint cliques of size $r$ in a complete $r$-partite graph.

\begin{thm}[Main Theorem]\label{main}
For any integers $1 \leq k \leq n_1 \leq ... \leq n_r$,
\[
\operatorname{ex}(K_{n_1, ... , n_r}, kK_r) = \sum_{1\leq i<j\leq r}n_in_j - n_1n_2 + n_2(k-1).
\]
\end{thm}
For the lower bound, clearly $((n_1-k+1)K_1\cup K_{k-1,n_2})+K_{n_3,\ldots,n_r}$ is such a subgraph of $K_{n_1,\ldots,n_r}$ with the required number of edges and no copy of $kK_r$. This gives the lower bound and the remainder of this paper will work to establish the upper bound. The upper bound is proven by considering two cases: $n_2=n_r$ and $n_2<n_r$. In the former case, the proof is inductive on $n_1+k$ with Lemmas \ref{base1} and \ref{base2} as the base cases. In the latter case, the proof is inductive on the total number of vertices in the host graph.


 \section{Main results}
The proof of Theorem \ref{main} first considers the case when the host graph is almost balanced, that is every part size is the same except the smallest part. This proof requires a double induction and is preceded by the two necessary base cases.
 
 Define
 \[h_k(n_1,n_2,\ldots,n_r)=\sum_{1\leq i<j\leq r}n_in_j-n_1n_2+n_2(k-1).\]
Given two disjoint subsets of the vertex set, $A,B\subseteq V(G)$, define $AB$ as the graph formed by the set of edges in $G$ incident to a vertex in $A$ and a vertex in $B$.
 
 For ease of notation, given an $r$-partite graph $G$ with parts $V_1,\ldots,V_r$ we let $\mathcal{R}(G,r)=\{\{v_1,\ldots,v_r\}\subseteq V(G)\,:\, v_i\in V_i\text{ for all }i\in[r]\},$ that is $\mathcal{R}(G,r)$ is the set of all $r$-tuples of vertices with exactly one vertex from each part. We utilize $\mathcal{R}(G,r)$ throughout to facilitate the counting of edges. For $S\in \mathcal{R}(G,r)$ define
 \[w(S)=|E(G[S])|.\]
 Note that for $S\in\mathcal{R}(G,r)$ an edge $v_iv_j\in V_iV_j$ is counted in $w(S)$ if and only if both $v_i$ and $v_j$ are present in $S$, therefore, summing over all $S\in\mathcal{R}(G,r)$,
 \begin{equation}\label{weight}\sum_{S\in\mathcal{R}(G,r)}w(S)=\sum_{1\leq i<j\leq r}|E(V_iV_j)|\prod_{\ell\neq i,j}n_\ell.\end{equation}
 
\begin{lem}\label{base1}
For $1 \leq n_1 \leq n_2,$
\[
\operatorname{ex}(K_{n_1,n_2,\ldots,n_2}, K_r) = h_1(n_1,n_2,\ldots,n_2).
\]
\end{lem}
\begin{proof}
Suppose $G \subseteq K_{n_1,n_2,...,n_2}$ does not contain a copy of $K_r$. Then for all $S\in\mathcal{R}(G,r)$, $w(S) \leq \binom{r}{2}-1$ and hence
\begin{equation}\label{ineq1}
\sum_{S\in\mathcal{R}(G,r)} w(S) \leq \left(\binom{r}{2}-1\right)n_1n_2^{r-1}.
\end{equation}
Subtracting (\ref{ineq1}) from (\ref{weight}) yields,
\begin{align*}
0 &\geq \sum_{j=2}^{r}|E(V_1V_j)|n_2^{r-2} + \sum_{i,j\neq 1} |E(V_iV_j)|n_1n_2^{r-3}-\left(\binom{r}{2}-1 \right)n_1n_2^{r-1} \\
&= n_1n_2^{r-3}|E(G)| + \sum_{j=2}^{r}|E(V_1V_j)|n_2^{r-3}(n_2-n_1)- \left(\binom{r}{2}-1 \right)n_1n_2^{r-1} \\
&\geq n_1n_2^{r-3}|E(G)|+\left(|E(G)|-\binom{r-1}{2}n_2^2\right)n_2^{r-3}(n_2-n_1)-\left(\binom{r}{2}-1\right)n_1n_2^{r-1}\\
&=n_2^{r-2}|E(G)|-\binom{r-1}{2}n_2^{r-1}(n_2-n_1)-\left(\binom{r}{2}-1\right)n_1n_2^{r-1}\\
&=n_2^{r-2}|E(G)|-(r-2)n_1n_2^{r-1}-\binom{r-1}{2}n_2^r.
\end{align*}
Therefore
\[|E(G)|\leq n_1n_2(r-1)+\binom{r-1}{2}n_2^2-n_1n_2=h_1(n_1,n_2,\ldots,n_2).\]
\end{proof}


\begin{lem}\label{base2}
For $1 \leq n_1 \leq n_2$,
$$
\operatorname{ex}(K_{n_1,n_2,...,n_2}, n_1K_r) = h_{n_1}(n_1,n_2,\ldots,n_2).
$$
\end{lem}

\begin{proof}
This proof is by induction on $n_1$. The base case of $n_1=1$ is true for all positive integers $n_2$ by Lemma \ref{base1}. Assume the statement holds for $n_1'<n_1$ where $n_1\geq 2$. Suppose $G\subseteq K_{n_1,n_2,\ldots,n_2}$ contains more than $h_{n_1}(n_1,n_2,\ldots,n_2)$ edges and does not contain a copy of $n_1K_r$. We have $$|E(G)|>h_{n_1}(n_1,n_2,\ldots,n_2)>h_1(n_1,n_2,\ldots,n_2),$$ which implies $G$ contains a copy of $K_r$. Let $S\in\mathcal{R}(G,r)$ such that $G[S]\cong K_r$. Then $|E(G\backslash S)|\leq h_{n_1-1}(n_1-1,n_2-1,\ldots,n_2-1),$ otherwise $G\backslash S$ contains a copy of $(n_1-1)K_r$, and this together with $G[S]$ is a copy of $n_1K_r$ in $G$. Therefore,
\begin{align*}
|E(G)|-|E(G\backslash S)|&>h_{n_1}(n_1,n_2,\ldots,n_2)-h_{n_1-1}(n_1-1,n_2-1,\ldots,n_2-1)\\
&=(r-1)(n_1+(r-2)n_2)+(r-1)n_2-\binom{r}{2}-1.
\end{align*}
Since the number of edges in $K_{n_1,n_2,\ldots,n_2}$ with a vertex in $S$ is \[(r-1)(n_1+(r-2)n_2)+(r-1)n_2-\binom{r}{2},\] this implies all edges in the host graph containing a vertex in $S$ are present in $G$. Note that this holds for every $S$ such that $G[S]\cong K_r$ in $G$.

Let $u_i\in V_i$ and $u_j\in V_j$ with $i\neq j$, if neither $u_i$ nor $u_j$ is in $S$, then $u_iu_j\in E(G)$. Otherwise, for $v_i\in S\cap V_i$, let $S'=(S\backslash\{v_i\})\cup\{u_i\}$. $S'$ induces a copy of $K_r$ in $G$ and therefore $u_iu_j\in E(G)$. Hence $G\cong K_{n_1,n_2,\ldots,n_2}$ and thus contains $n_1K_r$ which is a contradiction.
\end{proof}

We now prove Theorem \ref{main}. The proof considers two cases: $n_2=n_r$ and $n_2<n_r$. The first case is proven using double induction and relies on Lemmas \ref{base1} and \ref{base2} as the base cases. The second case is proven by induction on the total number of vertices in the host graph. 
\begin{proof} Let $1\leq k\leq n_1\leq\cdots\leq n_r$.

\emph{Case 1.} Assume $n_2=n_r$, we proceed by induction on $n_1 + k$. The base case of $k=1$ is true for all positive integers $n_1$ by Lemma \ref{base1} and the base case of $n_1=k$ is true for all positive integers $k\leq n_1$ by Lemma \ref{base2}. Assume the statement is true for parameters $n_1',k'$ such that $n_1'+k'<n_1+k$ where $n_1>k\geq 2$, and that $G \subseteq K_{n_1, n_2,\ldots,n_2}$ does not contain a copy of $kK_r$.

We first obtain a lower bound on the number of (not necessarily disjoint) copies of $K_r$ in $G$. Suppose there are exactly $q$ such copies of $K_r$ in $G$, then
\begin{align*}
\sum_{S\in \mathcal{R}(G,r)} w(S) &\leq q\binom{r}{2} + (n_1n_2^{r-1} - q)\left( \binom{r}{2} - 1 \right).
\end{align*}
Recall that \[\sum_{S\in\mathcal{R}(G,r)}w(S)=\sum_{j=2}^r |E(V_1V_j)|n_2^{r-2} + \sum_{i,j\neq 1} |E(V_iV_j)|n_1n_2^{r-3},\] this gives,
\begin{equation}\label{ineq2}q\geq \sum_{j=2}^r |E(V_1V_j)|n_2^{r-2} + \sum_{i,j\neq 1} |E(V_iV_j)|n_1n_2^{r-3} - n_1n_2^{r-1}\left( \binom{r}{2} - 1 \right).\end{equation}

We will use (\ref{ineq2}) to get an upper bound on $|E(G)|$ by counting $\displaystyle{\sum_{S\in\mathcal{R}(G,r)}|E(G\backslash S)|}$. An edge $v_iv_j\in V_iV_j$ is counted in $|E(G\setminus S)|$ if and only if $v_i\not\in S$ and $v_j\not\in S$, hence
{\footnotesize{\begin{equation}\label{ineq3}
\sum_{S\in \mathcal{R}(G,r)} |E(G \setminus S)| =  \sum_{j=2}^r |E(V_1V_j)|(n_1-1)(n_2-1)n_2^{r-2} + \sum_{i,j\neq 1} |E(V_iV_j)|(n_2-1)^2n_1n_2^{r-3}.
\end{equation}}}

Using (\ref{ineq2}) and (\ref{ineq3}),
{\footnotesize{
\begin{align}
\sum_{S\in\mathcal{R}(G,r)}|E(G\backslash S)|+\left(q+n_1n_2^{r-1}\left(\binom{r}{2}-1\right)\right)(n_2-1)&\geq \sum_{j=2}^r|E(V_1V_j)|n_1(n_2-1)n_2^{r-2}\nonumber\\&+\sum_{i,j\neq1}|E(V_iV_j)|(n_2-1)n_1n_2^{r-2}\nonumber\\&=|E(G)|(n_2-1)n_2^{r-2}n_1.\label{ineq4}
\end{align}
}}
Now for $S\in \mathcal{R}(G,r)$, suppose $G[S]$ is a copy of $K_r$. Then $|E(G\backslash S)|\leq h_{k-1}(n_1-1,n_2-1,\ldots,n_2-1)$ else by induction $G\backslash S$ contains a copy of $(k-1)K_r$ and so this together with $G[S]$ yields a copy of $kK_r$ in $G$. If $G[S]$ is not complete, then since $G\backslash S$ does not contain a copy of $kK_r$ induction gives $|E(G\backslash S)|\leq h_k(n_1-1,n_2-1,\ldots,n_2-1)$. Hence
\begin{align*}
\sum_{S\in \mathcal{R}(G,r)} |E(G \setminus S)| &\leq q \big( h_{k-1}(n_1-1, n_2-1,\ldots, n_2 - 1) \big) \\ 
	&+ (n_1n_2^{r-1} - q) \big( h_k(n_1-1,n_2-1,\ldots,n_2 - 1) \big) \\
	&= q(1-n_2) +  n_1n_2^{r-1}(h_k(n_1-1,n_2-1,\ldots,n_2-1))
\end{align*}
and thus, using (\ref{ineq4}), we have
{\footnotesize{
\[
|E(G)|(n_2-1)n_2^{r-2}n_1 \leq n_1n_2^{r-1}\left(h_{k}(n_1-1,n_2-1,\ldots,n_2-1)+\left(\binom{r}{2}-1\right)(n_2-1)\right).
\]
}}
Therefore
\begin{align*}
|E(G)| &\leq \frac{n_2}{n_2-1}\left(h_k(n_1-1,n_2-1,\ldots,n_2-1)+\left(\binom{r}{2}-1\right)(n_2-1)\right)\\
& = h_k(n_1,n_2,n_2,\ldots,n_2).
\end{align*}

\emph{Case 2.} Assume $n_2<n_r$, we proceed by induction on the number of total vertices. The base case of $n_1=n_r$ is true for all positive integers $k$ by Case 1. Assume the statement holds for all parameters $n_1',\ldots,n_r'$ such that $\sum_{i=1}^rn_i'<\sum_{i=1}^rn_i$. Suppose $G\subseteq K_{n_1,\ldots,n_r}$ does not contain a copy of $kK_r$. Let $v_r \in V_r$, the graph $G\backslash\{v_r\}$ does not contain a copy of $kK_r$, has fewer vertices than $G$, and $n_2\leq n_r-1$. Therefore
\begin{align*}
|E(G)| &= |E(G \setminus \{v_r\})| + d(v_r) \\
	&\leq \operatorname{ex}(K_{n_1,\ldots,n_r-1},kK_r) + d(v_r) \\
	&=h_k(n_1,\ldots,n_r-1)+d(v_r)\\
	& = \sum_{1\leq i<j\leq r}n_in_j - \sum_{i = 1}^{r-1}n_i - n_1n_2 + n_2(k-1) + d(v_r)\\
	&\leq \sum_{1\leq i<j\leq r}n_in_j - n_1n_2 + n_2(k-1) \\
	&= h_k(n_1,n_2,\ldots,n_r).
\end{align*}
\end{proof}


\section{Concluding remarks}

The main theorem relies on the fact that both $K_r$ and $K_{n_1,n_2,\ldots, n_r}$ are $r$-partite. Certainly the host graph must be $\ell$-partite for $\ell \geq r$ to have $K_r$ as a subgraph. An interesting generalization would be to calculate $\operatorname{ex}(K_{n_1,n_2,\ldots,n_\ell},kK_r)$ for $r<\ell$. In \cite{DHY}, De Silva, Heysse, and Young proved that 
\[\operatorname{ex}(K_{n_1,n_2,\ldots,n_\ell},kK_2)=(k-1)\left(\sum_{i=2}^\ell n_i \right),\]
however the Tur\'{a}n number is open for $r\geq 3$. The graph $$((n_1+n_2-k+1)K_1\cup K_{k-1,n_3})+n_4K_1$$ does not contain $kK_3$, hence
\[\operatorname{ex}(K_{n_1,n_2,n_3,n_4},kK_3)\geq (n_1+n_2+n_3)n_4+(k-1)n_3.\]
This construction can be easily generalized to $r$-partite graphs, but it is not clear that this is an extremal construction.

%
%
%

We also note that many of the results cited in this paper were originally considered in conjunction with the \emph{rainbow number} $\operatorname{rb}(G,H)$. For a graph $G$ and subgraph $H$, $\operatorname{rb}(G,H)$ is the minimum number of colors required to ensure that every edge coloring of $G$ with $\operatorname{rb}(G,H)$ colors has a rainbow copy of $H$ (where a subgraph is \emph{rainbow} if it has no two edges with the same color).  Often the rainbow number is proven via the analogous Tur\'{a}n number, and it would be interesting to see this work extended to the rainbow number $\operatorname{rb}(K_{n_1,\ldots, n_r},kK_r)$.

\section{References}

\bibliographystyle{model1-num-names}
\bibliography{sample.bib}

\begin{thebibliography}{00}
 \bibitem{BiK}
 H. Bielak, S. Kieliszek, The Tur\'{a}n number of the graph 3$P_4$, Annales Universitatis Mariae Curie-Sklodowska, sectio A -- Mathematica 68 (1) (2014) 21--29.
 
 \bibitem{BiK2}
 H. Bielak, S. Kieliszek, The Tur\'{a}n number of the graph 2$P_5$, Discuss. Math. Graph Theory 36 (3) (2016) 683--694.
 
 \bibitem{BM}
 J.A. Bondy, U.S.R. Murty, Graph Theory with Applications, Macmillan, London, 1976, Elsevier, New York.
 
 \bibitem{BK}
 N. Bushaw, N. Kettle, Tur\'{a}n numbers of multiple paths and equibipartite forests, Combinatorics, Probability and Computing, 20 (2011) 837--853.
 
  \bibitem{CTL}
H. Chen, X. Li, J. Tu, Complete solution for the rainbow numbers of matchings,  Discrete Math. 309 (10) (2009) 3370--3380.
 
 \bibitem{DHY}
 J. De Silva, K. Heysse, M. Young, Rainbow number for matchings in $r$-partite graphs, preprint. 
 
 \bibitem{Gor}
I. Gorgol, Tur\'{a}n numbers for disjoint copies of graphs, Graphs and Comb. 27 (5) (2011) 661--667.
 
\bibitem{LTJ}
X. Li, J. Tu, Z. Jin, Bipartite rainbow numbers of matchings, Discrete Math. 309 (2009) 2575--2578.

\bibitem{LLP}
B. Lidick\'{y}, H. Liu, C. Palmer, On the Tur\'{a}n number of forests, Electr. J. Comb. 20 (2) (2012) $\#$P62.

\bibitem{LT2}
L-T.\,Yuan, X-D.\,Zhang, Tur\'{a}n numbers for disjoint paths,\\arXiv:1611.00981.

\bibitem{YZ}
L-T.\,Yuan, X-D.\,Zhang, The Tur\'{a}n number of disjoint copies of paths, Discrete Math. 340 (2) (2017) 132--139.
 \end{thebibliography}

\end{document}